\documentclass[11pt,reqno]{amsart}
\usepackage{amsmath,amsfonts,amssymb,amscd,amsthm,amsbsy,epsf}
\newtheorem{thm}{Theorem}
\newtheorem*{cor}{Corollary}
\newtheorem{Cor}{Corollary}

\theoremstyle{definition}

\def\F{{\mathcal F}}
\def\S{{\mathcal S}}
\def\real{{\mathbb R}}
\def\sumprime{\mathop{{\sum}'}}
\def\Prod{\textstyle\prod}
\begin{document}
\title[Multilinear Embedding]
{Multilinear Embedding -- convolution estimates  on smooth  submanifolds}
\author{William Beckner}
\address{Department of Mathematics, The University of Texas at Austin,
1 University Station C1200, Austin TX 78712-0257 USA}
\email{beckner@math.utexas.edu}
\begin{abstract}
Multilinear embedding estimates for the fractional Laplacian are obtained in terms of 
functionals defined over a hyperbolic surface.
Convolution estimates used in the proof enlarge the classical framework of the convolution 
algebra for Riesz potentials to include the critical endpoint index, and provide new 
realizations for   fractional integral  
inequalities that incorporate restriction to smooth submanifolds. 
Results developed here are modeled on the space-time estimate used by Klainerman and 
Machedon in their proof of uniqueness for the Gross-Pitaevskii hierarchy.
\end{abstract}
\maketitle

Analysis of the Gross-Pitaevskii hierarchy has led to the development and 
application of functional analytic mappings  for the rigorous description of many-body
interactions in quantum dynamics. 
Evolving  from this framework is increased understanding for how Sobolev embedding 
and measures of fractional smoothness determine intrinsic size and growth estimates 
for functions and their Fourier transforms. 
Development of multilinear analysis increases understanding for genuinely $n$-dimensional 
aspects of Fourier analysis.
In a formative and influential paper on uniqueness of solutions for 
the Gross-Pitaevskii hierarchy (\cite{KM}), Klainerman and Machedon 
prove a novel space-time estimate where the essential part of the proof corresponds to having 
uniform bounds for a three-dimensional convolution integral taken over a hyperbolic surface.
Their result can be interpreted in the larger context of multilinear embedding where new 
end-point estimates are obtained.
The resulting inequalities can be viewed as a step in the larger and dual program for 
understanding how smoothness controls restriction to a non-linear sub-variety 
(see \cite{Beckner-MRL})

The paradigm that underlies this objective to characterize control by multilinear embedding 
combines aspects of the Hardy-Littlewood-Sobolev inequality, 
the Hausdorff-Young inequality, Sobolev embedding  and the uncertainty principle:
$$\int_{\real^n} |\hat f|^2\,d\xi = \int_{\real^n} |f|^2\,dx 
\le C \bigg[ \int_{\real^n} \Big| (-\Delta/4\pi^2)^{\alpha/2} f\Big|^p \,dx\bigg]^{2/p}$$
where $\alpha = n (1/p - 1/2) \ge 0$ and $1<p\le 2$. 
Our objective   is to obtain multilinear embedding forms that extend this inequality and 
express capability for fractional smoothness to control restriction on a smooth submanifold.
Results in this direction have already been given in \cite{Beckner-MRL}, and initially 
reflect ideas of Calder\'on and Stein. 
The possibility that arises from  the Klainerman-Machedon space-time estimate may be 
implicitly suggested by Stein's   observation that ``surface restriction'' for the 
Fourier transform can sharpen estimates that use fractional integral arguments 
(see page 28 in \cite{F}; pages 352--353, 374 in \cite{Stein93}): 
that is, surface integrals can be used as an auxiliary mechanism 
to characterize embedding action by Riesz potentials:
\begin{equation}\label{eq:Riesz-embedding}
\int_S |x_1 +\cdots + x_m|^\sigma \Big|\F \big[ \Prod |x_k|^{-\beta_k} * f\big] (x_1,\ldots, x_m) 
\Big|^r\,d\nu\ .
\end{equation}
This is not directly ``restriction phenomena'' but rather a novel domain decomposition 
that results from adding new variables and submanifold restriction occurs for interior 
potential calculations used to define embedding forms.
Still estimates of this kind have been used for restriction-related arguments 
(see page 204 in \cite{KM96}). 
The choice of a hyperbolic surface for this functional reflects both application-driven 
problems (free Schr\"odinger equation, Coulomb forces) and geometric  invariance 
(conformal group,  indefinite orthogonal group). 
The existence of distinguished directions for the surface will place limits on the range of 
multilinear embeddings that are considered and make the selection of uniform potentials 
and parameter constraints such as $m=n$  more intrinsic. 
This strategy will reinforce the underlying purpose for the paradigm: 
``{\em symmetry determines structure}''. 
\bigskip
\subsection*{Outline of argument:}$\quad$
\bigskip

\noindent
Consider $m$ copies of $\real^n$ and let $f$ be in the Schwartz class $\S (\real^{mn})$.
Define the Fourier transform 
$$(\F f)(\xi) = \hat f(\xi) = \int e^{2\pi i\xi x} f(x)\,dx\ ,$$
and observe that on $\real^n$ with $0<\lambda <n$
$$\F : |x|^{-\lambda}\quad \longrightarrow\quad \pi^{-n/2 +\lambda} 
\left[ \frac{\Gamma (\frac{n-\lambda}2)}{\Gamma (\frac{\lambda}2)}\right] 
|x|^{-n+\lambda}\ .$$
For   $\Delta_k =$ standard Laplacian on $\real^n$ in the variable $x_k$, 
$\{\alpha\} = (\alpha_1,\alpha_2,\ldots,\alpha_m)$, 
$0<\alpha_k <n$, $\alpha = \sum \alpha_k$ for $k=1$ to $m$, 
$1<p$, $1< p_* \le 2$ and define 
\begin{equation}\label{eq:outline}
\begin{split}
&\Lambda_{p_*}  (f;\{\alpha\}/ p)   
= \int_{\real^n\times\cdots\times\real^n} 
\Big|  \Prod\limits_{k=1}^m 
(-\Delta_k/4\pi^2)^{\alpha_k/2p} f\Big|^{p_*}\, 
dx_1\ldots dx_m\\
\noalign{\vskip6pt}
&= \int_{\real^n\times\cdots\times\real^n} 
\Big| \int e^{-2\pi ix\xi} \Prod\limits_{k=1}^m |\xi_k|^{\alpha_k/p} \ \hat f  d\xi_1\ldots d\xi_m\Big|^{p_*}
dx_1 \ldots dx_m\ .
\end{split}
\end{equation}
For $w\in \real^n$, $\tau>0$ define 
$$d\nu = \delta \Big( \tau + \sumprime |x_k|^2 - |x_m|^2\Big) \delta \Big(w-\sum x_k\Big)\, 
dx_1 \ldots dx_m$$
(here the prime on the symbol for sum, product or sequence indicates that the 
last term should be dropped);
then for $\sigma >0$, $r\ge 1$, $1<p<\infty$, $1/p + 1/q = 1$, $rq \ge 2$, 
$1/p_* + 1/(rq) =1$  and $\beta_k = n-\alpha_k/(pr)$
\begin{equation} \label{eq:A}
\begin{split}
&A \int \bigg[ \int |x_1 + \cdots + x_m|^{\sigma/p} \Big|\F\Big[ \Prod |x_k|^{-\beta_k} * f\Big] 
(x_1,\ldots,x_m)\Big|^r\,d\nu \bigg]^q\,dw\,d\tau\\
\noalign{\vskip6pt}
&\qquad
 = \int \bigg[ \int |x_1 + \cdots + x_m|^{\sigma/p} \Prod |x_k|^{-\alpha_k/p} |\hat f|^r\,d\nu
\bigg]^q\,dw\, d\tau\\
\noalign{\vskip6pt}
&\qquad
\le \sup_{w,\tau} \bigg[ \int |x_1 + \cdots + x_m|^\sigma \Prod |x_k|^{-\alpha_k} \,d\nu\bigg]^{q/p}
\int \bigg[ \int |\hat f|^{rq} \,d\nu\bigg]\, dw\,d\tau\\
\noalign{\vskip6pt}
&\qquad
\le C \int |\hat f|^{rq} \,dx \le C\bigg[ \int  |f|^{p_*}\,dx\bigg]^{rq/p_*}\ .
\end{split}
\end{equation}
This result can be rephrased in terms of multilinear embedding: 
\begin{equation}\label{eq:multilinear-embed}
\bigg[\int \bigg[ \int |x_1 +\cdots + x_m|^{\sigma/p}  |\hat f|^r\,d\nu\bigg]^q \,dw\,d\tau
\bigg]^{p_*/(rq)} 
\le \Lambda_{p_*}  \Big(f; \{\alpha\}/( pr)\Big)\ .
\end{equation}
Observe that for $p_* =2$ and $pr \ge2$ (equivalently $p\ge 2)$, the range of values 
for $\alpha = \sum \alpha_k$ can be lowered from the limit used in \cite{Beckner-MRL}. 
The essential step for this argument is to determine the range of parameter values for which 
\begin{equation}\label{eq3}
\sup_{w,\tau} \int |x_1 +\cdots  + x_m|^\sigma \Prod |x|^{-\alpha_k}\,d\nu
\end{equation}
will be bounded.

Such results will extend the classical convolution for Riesz potentials 
$$\int_S \frac1{|w-y|^\lambda}\ \frac1{|y|^\mu}\, dv $$
where $S$ is a smooth submanifold in $\real^n$, $w\in \real^d$ and the objective 
is to bound the size of the integral by an inverse power of $|w|$ under suitable conditions 
on $\lambda$ and $\mu$. 
And in turn, these bounded convolution forms can serve as kernels to define new 
Stein-Weiss fractional integrals that characterize control by smoothness.
Formula~\eqref{eq3} is equivalent to 
$$\sup_{w,\tau} |w|^\sigma \int \delta \Big( \tau + \sumprime |x_k|^2 - |x_m|^2\Big) 
\delta \Big( w -\sum x_k\Big)  \Prod |x_k|^{-\alpha_k} \, dx_1 \ldots dx_m\ .$$
In the case where $\sigma = 2 + \alpha -n (m-1)$, this form has an invariance 
which makes it a function of only one variable:
\begin{equation}\label{eq:function}
\sup_w |w|^\sigma \int \delta \Big( 1+\sumprime |x_k|^2 - |x_m|^2\Big) 
\delta \Big( w-\sum x_k\Big) \Prod |x_k|^{-\alpha_k}\,dx_1 \ldots dx_m\ .
\end{equation}
Now for the case $\alpha = n (m-1)$, this form becomes an extension of the classical 
convolution form 
$$(g * f_1 * \cdots * f_m) (w)\ ,\qquad g\in L^1 (\real^n)\ ,\qquad f_k \in L^{n/\alpha_k} 
(\real^n)$$
which is uniformly continuous and in the class $C_o (\real^n)$ by using the Riemann-Lebesgue 
lemma for convolution. 
Here the convolution for Lebesgue classes is replaced by Riesz potentials, but the 
multivariable integration is constrained to be on a hyperbolic surface invariant under the action
of the indefinite orthogonal group.  

Effectively one obtains a new embedding kernel from formula~\eqref{eq:function} 
\begin{equation}\label{eq:new-kernel}
\int \delta \Big( 1+\sumprime |x_k|^2 - |x_m|^2\Big) \delta \Big( w-\sum x_k\Big) 
\Prod |x_k|^{-\alpha_k}\, dx_1 \ldots dx_m 
\le \frac{C}{|w|^2}
\end{equation} 
which defines a map from $L^{n/(n-1)} (\real^n)$ to $L^n (\real^n)$. 
Any type of efficient proof for estimates of this type will require a reduction argument for 
the number of convolutions which in turn will place constraints on the range of possible values 
of parameters. 
Though at first glance seemingly specialized, the case of uniform potentials with $m=n$ 
captures essential features of this problem, including the 3-dimensional model 
calculation by Klainerman and Machedon, with small allowance for variation. 
In part, this dominant feature reflects the conversion of the required multivariable estimate 
to basically an 
$n$-dimensional calculation in terms of vector lengths:
$$\Prod |x_k|^{-(n-1)} \,dx_1 \ldots dx_n = d|x_1|\ldots d|x_n|\, d\sigma$$
where $d\sigma$ denotes surface measure on $n$ copies of $S^{n-1}$. 
Moreover, such a choice allows a natural reduction to low dimension when the number 
of convolutions  is at least three. 
The strategy for developing this argument rests on three cornerstones:
1)~reduction of number of convolutions in the estimate;
2)~direct integration methods with focus on utilization of polar angle integration to obtain 
sharper dependence on the ``free variable'' $|w|$;
3)~explicit calculations in two dimensions. 
Notice that the estimates \eqref{eq:function}  for the surface integrals do not depend on the 
embedding indices for Lebesgue classes. 

\begin{thm}\label{thm:alpha}
Let $0<\alpha_k <n$, $\alpha = \sum \alpha_k$, $m\ge 3$, $n\ge 3$ and
$\rho = 2+\alpha - (m-1)n$ for $0<\rho <n$. 
For $m>\ell\ge2$ assume that there are $\ell $ choices of the $\alpha_k$'s having the 
property that $\ell n -2>\beta_{\ell,m} > (\ell-1) n-2$ with $\beta_{\ell,m} = \sum\alpha_k$
for $k= m-\ell+1$ to $m$.
Then for $\sigma = 2+\beta_{\ell,m} - (\ell-1)n$
\begin{equation}\label{eq:alpha} 
\begin{split}
&\sup_w |w|^\rho \int \delta \Big( 1+\sumprime |x_k|^2 - |x_m|^2\Big) 
\delta \Big( w-\sum x_k\Big) \Prod |x_k|^{-\alpha_k}\,dx_1 \ldots dx_m\\
\noalign{\vskip6pt}
&\le C \sup_w |w|^\sigma \int \delta \Big( 1+\sumprime |x_k|^2 - |x_m|^2\Big)
\delta \Big( w-\sum x_k\Big) \Prod |x_k|^{-\alpha_k} \, dx_{m-\ell+1} \ldots dx_m
\end{split}
\end{equation}
where the sums and products in the second term are taken over the indices 
$k= m-\ell+1$ to $m$.
\end{thm}

\begin{proof} 
\begin{equation*}
\begin{split}
&|w|^\rho \int \delta \Big( 1 +\sumprime |x_k|^2 - |x_m|^2\Big) 
\delta \Big( w-\sum x_k\Big) \Prod |x_k|^{-\alpha_k} \, dx_1 \ldots dx_m\\
\noalign{\vskip6pt}
&\qquad =  |w|^\rho \int {\Prod\limits_1^{m-\ell} }|x_k|^{-\alpha_k} 
\Big| w - \sum_1^{m-\ell} x_k\Big|^{-\sigma} 
\bigg[ \Big| w- \sum_1^{m-\ell} x_k\Big|^\sigma
 \int \delta \Big( 1+\sumprime |x_k|^2  - |x_m|^2\Big) 
 \delta \Big( w-\sum x_k\Big)\ \ \times\\
 &\qquad\qquad {\Prod\limits_{m-\ell+1}^m} |x_k|^{-\alpha_k} \, 
 dx_{m-\ell+1} \ldots dx_m\bigg] \, dx_1 \ldots dx_{m-\ell}\\
 \noalign{\vskip6pt}
&\qquad \le |w|^\rho \int {\Prod\limits_1^{m-\ell} } 
|x_k|^{-\alpha_k} \Big| w-\sum_1^{m-\ell} x_k\Big|^{-\sigma} \,dx_1 \ldots dx_{m-\ell}\ \ \times\\
 \noalign{\vskip6pt}
&\qquad\qquad \sup_{w,\tau} \bigg[ |w|^\sigma \int \delta 
\Big( \tau + \sumprime_{m-\ell+1} |x_k|^2 - |x_m|^2\Big) 
\delta \Big( w - \sum_{m-ell+1} x_k\Big) 
{\Prod\limits_{m-\ell+1} } |x_k|^{-\alpha_k} \, dx_{m-\ell+1} \ldots dx_m\bigg]\\
 \noalign{\vskip6pt}
 &\qquad =  C \sup_{w,\tau} \bigg[ |w|^\sigma \int \delta 
 \Big( \sumprime_{m-\ell+1} |x_k|^2 - |x_k|^2\Big) 
 \delta \Big( w - \sum_{m-\ell+1} x_k\Big) {\Prod\limits_{m-\ell+1}} 
 |x_k|^{-\alpha_k} \, dx_{m-\ell+1} \ldots dx_m\bigg] 
 \end{split}
 \end{equation*}
 where 
 $$C = |w|^\rho \int {\Prod\limits_1^{m-\ell} } 
 |x_k|^{-\alpha_k} \Big| w-\sum_1^{m-\ell} x_k\Big|^{-\sigma} \,dx_1 \ldots dx_{m-\ell}$$
 and can be calculated from the action of the Fourier transform on Riesz potentials.
 For the latter ``sup'' term, the new parameters $w,\tau$ were obtained by setting 
 $\tau = 1+\sum |x_k|^2$ for $1\le k\le m-\ell$ and $w-\sum x_k$  
 $(1\le k\le m-\ell)   \rightsquigarrow w$.  
 It must be recognized that the conditions given in the statement of the theorem may not 
 be sufficient to guarantee that the ``sup'' over $\ell$ convolutions is finite.
 \end{proof}
 
 The concern here will not be to treat the full range of parameters but to especially 
 cover the cases $\ell=2$ and 3, and to establish estimates for uniform potentials 
 with $\alpha_k = n-1$ for $3\le m\le n+1$.

\begin{thm}\label{thm-uniform-bd}
For $n\ge 2$,  $\sigma = \alpha +\lambda +2 -n$, $2(n-1)>\alpha +\lambda >n-1$
and $0<\alpha < n-1$, $\lambda >0$
\begin{equation}\label{eq:uniform-bd}
\Theta_n  (w) 
= |w|^\sigma \int_{\real^n\times\real^n} 
\delta \left[ 1 + |x|^2 - |y|^2\right]
\delta (w-x-y) |x|^{-\alpha} |y|^{-\lambda} \,dx\,dy 
\end{equation}
is uniformly bounded for $w\in \real^n$.
\end{thm}

\begin{proof}
Observe that since $|y| \ge 1$, there is no upper bound for $\lambda$ in this computation; 
but this calculation may be used for embedding where $0<\sigma <n$ so there can be 
an effective upper bound resulting from application. 
\begin{equation}\label{eq:pf-uniform-bd}
\begin{split}
\Theta_{n,\alpha} (w) 
& = |w|^\sigma \int_{\real^n} \delta \left[ 1+ |w-y|^2 - |y|^2\right] |w-y|^{-\alpha} 
|y|^{-\lambda} \,dy \\
\noalign{\vskip6pt}
& = \frac{2\pi^{(n-1)/2}}{\Gamma ((n-1)/2)} |w|^\sigma 
\int_0^1\!\!\int_1^\infty  \delta \left( 1+|w|^2 - 2|w| ru\right) (r^2-1)^{-\alpha/2} 
r^{n-\lambda-1} (1-u^2)^{(n-3)/2} \,dr\,du\\
\noalign{\vskip6pt}
& = \frac{2^{\alpha +\lambda-n}\pi^{(n-1)/2}}{\Gamma ((n-1)/2)} 
\left[ \frac{|w|^2}{1+|w|^2}\right]^{\alpha +\lambda -n+1} 
\int_0^1 u^{(\alpha +\lambda -n+1)/2-1} 
(1-u)^{(n-3)/2} \big( 1-\beta (w)u\big)^{-\alpha/2}\,du\\
\noalign{\vskip6pt}
&= C_n \left[ \frac{|w|^2}{1+|w|^2}\right]^{\alpha +\lambda -n+1} 
F\big( \alpha/2 ,\ (\alpha+\lambda -n+1)/2;\ (\alpha +\lambda)/2;\ \beta(w)\big)
\end{split}
\end{equation}
where $F$ denotes the hypergeometric function, 
$\beta (w) = \frac{4|w|^2}{(1+|w|^2)^2}\le 1$
and 
$$C_n =  2^{\alpha +\lambda -n} \pi^{(n-1)/2} \ 
\frac{\Gamma ((\alpha + \lambda -n+1)/2)}{\Gamma ((\alpha +\lambda)/2)}\ .$$
Then 
\begin{equation}\label{eq:pf-uniform-bd-2}
\theta_n(w) \le 2^{\alpha +\lambda -n} \pi^{(n-1)/2} 
\frac{\Gamma ((\alpha +\lambda -n+1)/2)\Gamma ((n-1-\alpha)/2)}{\Gamma((n-1)/2)
\Gamma (\lambda/2)} 
\left[ \frac{|w|^2}{1+|w|^2}\right]^{\alpha+\lambda -n+1}\ .
\end{equation}
Hence $\Theta_n(w)$ is bounded for $\lambda >0$, $0<\alpha <n-1$ 
and $\alpha +\lambda >n-1$.
\end{proof}

This result demonstrates that at least three convolutions are needed to 
account for uniform potentials with $\alpha_k = n-1$. 
In the case of dimension two, one can explicitly compute that 
\begin{equation*}
\begin{split}
& |w|^2 \int_{\real^n \times\real^2} \delta (1+|x|^2 - |y|^2) 
\delta (w- x-y) |x|^{-1} |y|^{-1}\,dx\,dy\\
\noalign{\vskip6pt}
&\qquad = 
\frac{|w|^2}{1+|w|^2} \int_0^1 u^{-1/2} (1-u)^{-1/2} \big( 1-\beta (w)u\big)^{-1/2} \,du\\
\noalign{\vskip6pt}
&\qquad = 
\frac{2|w|^2}{1+|w|^2} K\Big( \sqrt{\beta (w)}\Big) 
\simeq - \ln \sqrt{ 1-\beta (w)} \ \text{ for }\ \beta (w) \simeq 1
\end{split}
\end{equation*}
where $K$ denotes the complete elliptic integral. 

\begin{thm}\label{thm:delta-function}
For $n\ge2$, $\sigma = 2+\alpha_1 + \alpha_2 +\lambda - 2n$ and 
$3n-2 > \alpha_1 +\alpha_2 +\lambda >2n-2$, 
$\alpha_1 + \lambda > n-1$, $\alpha_2 + \lambda > n-1$, $2n-1 >\alpha_1 + \alpha_2$
\begin{equation}\label{eq:delta-function}
\Delta_n (w) = |w|^\sigma \!\! \int_{\real^n\times\real^n\times\real^n} \mkern-18mu
\delta \left[ 1+ |w|^2 + |z|^2 - |y|^2\right] 
\delta (w-x-y-z) |x|^{-\alpha_1} |z|^{-\alpha_2} |y|^{-\lambda}\,dx\,dy\,dz
\end{equation}
is uniformly bounded for $w\in \real^n$.
\end{thm}

\begin{proof} 
Using translations and applying the two delta functions
\begin{equation*}
\begin{split} 
\Delta_n (w) & = |w|^\sigma \int_{\real^n\times\real^n\times\real^n} \mkern-22mu
\delta \left[ 1+ |w|^2 + |z|^2 - |y|^2\right] 
\delta (x+y+z) |x+w|^{-\alpha_1} |z|^{-\alpha_2} |y|^{-\lambda} \,dx\,dy\,dz\\
\noalign{\vskip6pt}
& = |w|^\sigma \int_{\real^n\times\real^n} \mkern-22mu
\delta \left[ 1+ |x-w|^2 + |x-y|^2 - |y|^2\right] |x-w|^{-\alpha_1} |x-y|^{-\alpha_2} \, dx\,dy\,dz\\
\noalign{\vskip6pt}
& = |w|^\sigma \int_{\real^n\times\real^n} \mkern-22mu
\delta \left[ 1+ |x-w|^2 + |x|^2 - 2x\cdot y\right] |x-w|^{-\alpha_1} \big| |y|^2 - 1 - |x-w|^2
\big|^{-\alpha_2/2} |y|^{-\lambda}\,dx\,dy\\
\noalign{\vskip6pt}
& =     
 s_n |w|^\sigma
\int_{\real^n} \! \int_0^1 \!\! \int_1^\infty \mkern-12mu 
\delta \left[ 1+ |x-w|^2 + |x|^2 - 2|x| ur\right] |x-w|^{-\alpha_1} 
\big|r^2 - 1 - |x-w|^2\big|^{-\alpha_2/2} \ \ \times\\
\noalign{\vskip6pt}
&\qquad\hskip2.5truein
r^{n-\lambda-1} (1-u^2)^{(n-3)/2} \,dr\,du\,dx\\
\noalign{\vskip6pt}
& = 	 
s_n |w|^\sigma
\int_{\real^n} \! \int_0^1   
(2|x|u)^{\lambda -n+\alpha_2} 
\left[ 1+ |x-w|^2 + |x|^2\right]^{n-\lambda -1-\alpha_2} |x-w|^{-\alpha_1} \ \ \times\\
\noalign{\vskip6pt}
&\qquad\hskip1.5truein
\left[ 1- \frac{4|x|^2 (1+|x-w|^2)u^2}{(1+|x-w|^2 + |x|^2)^2} \right]^{-\alpha_2/2} 
(1-u^2)^{(n-3)/2}\,du\,dx\\
\noalign{\vskip6pt}
& = 	 
2^{\lambda -n + \alpha_2 -1} s_n |w|^\sigma 
\int_{\real^n} |x|^{\lambda +\alpha_2 -n} |x-w|^{-\alpha_1} 
\left[ 1+ |x-w|^2 + |x|^2\right]^{n-1-\lambda -\alpha_2}\ \ \times\\
\noalign{\vskip6pt}
&\qquad\hskip1truein 
\bigg[ \int_0^1 u^{(\lambda +\alpha_2 - n+1)/2-1} 
(1-u)^{(n-3)/2} \big( 1-\beta (x,w)u\big)^{-\alpha_2/2} \,du \bigg]\,dx
\end{split}
\end{equation*}
where  $s_n = 2\pi^{(n-1)/2}/\Gamma ((n-1)/2)$ 
$$\beta (x,w) = \frac{4|x|^2 (1+|w-x|^2)}{(1+ |w-x|^2 + |x|^2)^2}\ ,\qquad 
1-\beta (x,w) = \left[ \frac{1+|w-x|^2 - |x|^2}{1+ |w-x|^2 + |x|^2}\right]^2\ .$$
Then 
\begin{equation*}
\begin{split}
\Delta_n (w) & = 2^{\lambda + \alpha_2 -n-1} 
s_n \frac{\Gamma ((\lambda +\alpha_2 -n+1)/2)}{\Gamma((\lambda+\alpha_2)/2)} 
|w|^\sigma\ \ \times\\
\noalign{\vskip6pt}
&\qquad \int_{\real^n} |x|^{\lambda +\alpha_2 -n} |x-w|^{-\alpha_1} 
\Big[ 1+ |x-w|^2 + |x|^2\Big]^{n-1-\lambda - \alpha_2}\ \ \times\\
\noalign{\vskip6pt}
&\qquad\qquad {}_2F_1 \Big( \alpha_2/2,\ (\lambda+\alpha_2 -n+1)/2\ ;\ 
(\lambda +\alpha_2)/2\ ;\ \beta (x,w)\Big)\,dx\ .
\end{split}
\end{equation*}
\medskip

\noindent {\em Step 1}: 
The original calculation here is symmetric in $\alpha_1$ and $\alpha_2$. 
Suppose that for one of these values, say $\alpha_2$, that $\lambda + \alpha_2 >n-1$ 
and $\alpha_2 < n-1$.
The integral in $u$ above is bounded by 
$$\frac{\Gamma ((\lambda +\alpha_2 - n+1)/2) \Gamma ((n-1-\alpha_2)/2)}
{\Gamma (\lambda/2)}\ ;$$
then to see that $\Delta_2$ is uniformly bounded in $w$, it suffices to show that 
$$|w|^\sigma \int_{\real^n} |x|^{\lambda -n +\alpha_2} |x-w|^{-\alpha_1} 
\left[ 1+ |x-w|^2 + |x|^2\right]^{n-\lambda -1-\alpha_2}\,dx$$
is bounded. 
This term is less than 
$$|w|^\sigma \int_{\real^n} |x-w|^{-\alpha_1} 
\left[ 1+ |x-w|^2 + |x|^2\right]^{-(\lambda +\alpha_2 - n)/2-1}\,dx$$
which is bounded since the sum of the ``powers of $|x|$'' is less than $2n$ (by hypothesis) 
$$\alpha_1 + \alpha_2 +\lambda -n+2 < 2n$$
so that 
\begin{gather*}
|w|^\sigma \int_{\real^n} |x-w|^{-\alpha_1} 
\left[ 1+  |x-w|^2 + |x|^2\right]^{(-\lambda +\alpha_2 - n)/2-1}\, dx\\
\le |w|^\sigma \int_{\real^n} |x-w|^{-\gamma_1} |x|^{-\gamma_2}\,dx = C
\end{gather*}
where $0<\gamma_1,\gamma_2 <n$ and $\gamma_1 + \gamma_2 = \alpha_1 +\alpha_2
+ \lambda -n+2$.
\medskip

\noindent {\em Step 2}: 
Assume that $\alpha_2 \ge n-1$ (thus insuring $\lambda + \alpha_2 >n-1$), 
$\alpha_1 +\lambda > n-1$,  and $2n-1 >\alpha_1 + \alpha_2$;  
 choose $\gamma$ so that 
$$\min \Big\{ n-\alpha_1, \alpha_1 +\alpha_2 +\lambda - 2(n-1),1\Big\}
> 2\gamma > \max \Big\{ 1-\alpha_2, \alpha_2 - (n-1)\Big\} \ge 0\ ;$$
apply the estimate 
$$\big( 1-\beta (x,w)u\big)^{-\alpha_2/2} 
\le \big( 1-\beta (x,w)\big)^{-\gamma} (1-u)^{-(\alpha_2 - 2\gamma)/2}$$
and the integral in $u$ becomes 
\begin{equation*}
\begin{split} 
&\int_0^1 u^{(\lambda +\alpha_2 - n+1)/2 -1} 
(1-u)^{(n-1-\alpha_2 + 2\gamma)/2-1}\,du \\
\noalign{\vskip6pt}
&\qquad = \frac{\Gamma((\lambda +\alpha_2 -n+1)/2) \Gamma ((n-1+2\gamma-\alpha_2)/2)}
{\Gamma ((\lambda +2\gamma)/2)}\ .
\end{split}
\end{equation*}
For $\Delta_n (w)$ to be bounded, the form 
$$|w|^\sigma \int_{\real^n} |x|^{\lambda -n+\alpha_2} |x-w|^{-\alpha_1} 
\left[ 1+ |x-w|^2 + |x|^2\right]^{n-\lambda -1 -\alpha_2 +2\gamma} 
\left| 1+ |x-w|^2 - |x|^2\right|^{-2\gamma}\, dx$$
must be bounded, and this term is less than 
$$|w|^\sigma \int_{\real^n} |x-w|^{-\alpha_1} 
\left[ 1+ |x-w|^2 + |x|^2\right]^{-(\lambda + \alpha_2 -n)/2 + 2\gamma-1} 
\left| 1+ |x-w|^2 - |x|^2\right|^{-2\gamma}\,dx\ .$$
\renewcommand{\qed}{}
\end{proof}
\medskip

Choose the $x_1$ direction to be along that for $w_1$ and rearrange in the 
variable $x_1$ using for $x= (x_1,x') \in \real^n$
$$\int_{\real^n} f(x_1,x') g(x_1,x') h(x_1,x')\,dx_1\,dx'
\le \int_{\real^n} f_{\#} (x_1,x') g_{\#} (x_1,x') h_{\#} (x_1,x')\,dx_1 \, dx'$$
where $f_{\#} (x_1,x')$ is the equimeasurable decreasing rearrangement of 
$|f(x_1,x')|$ in the variable $x_1 \in \real$. 
After applying this rearrangement argument and making a change of variables 
to remove the dependence on $|w|$, the integral above is less than
\begin{equation*}
\begin{split} 
& |w|^{\sigma- 2\gamma} \int_{\real^n} |x|^{-\alpha_1} 
\Big[ 1+ 2|x|^2 + |w|^2/2\Big]^{-(\lambda +\alpha_2-n)/2\, + 2\gamma-1} 
|x_1|^{-2\gamma} \,dx\\
\noalign{\vskip6pt}
&\qquad  < \int_{\real^n} |x|^{-\alpha_1} \Big[ 2|x|^2 + 1/2\Big]^{-(\lambda + \alpha_2-n)/2\, +
2\gamma-1} |x_1|^{-2\gamma}\,dx\\
\noalign{\vskip6pt}
&\qquad  = \int_{\real^n} \Big(1+|x_1|^2\Big)^{-\alpha_1/2} |x_1|^{-2\gamma} 
|x'|^{1-\alpha_1-2\gamma} 
\Big[ 2|x'|^2 \Big( 1+|x_1|^2\Big) +1/2\Big]^{-(\lambda +\alpha_2 -n)/2\, + 2\gamma_1}\, dx\\
\noalign{\vskip6pt}
&\qquad = \int_{\real} \Big( 1+|x_1|^2\Big)^{-(n-2\gamma)/2} |x_1|^{-2\gamma}\,dx_1 
\int_{\real^{n-1}} |x'|^{1-\alpha_1-2\gamma} 
\Big[ 2|x'|^2 + 1/2\Big]^{-(\lambda + \alpha_2-n)/2\, + 2\gamma-1}\,dx'
\end{split}
\end{equation*}
For the given assumption on $2\gamma$, these integrals are bounded thus insuring 
that $\Delta_n(w)$ is uniformly bounded and completes the proof of 
Theorem~\ref{thm:delta-function}.\qed

Taken together Theorems~\ref{thm:alpha} and \ref{thm:delta-function} include both a 
natural extension and an independent proof of the Klainerman-Machedon 
three-dimensional estimate: 
\begin{gather}
|w|^2 \int_{\real^n\times\cdots\times\real^n} 
\delta \Big[ \tau + \mathop{{\sum}'} |x_k|^2 - |x_n|^2\Big] 
\delta \Big(w-\sum x_k\Big) {\Prod} |x_k|^{-(n-1)} 
dx_1\cdots dx_n \label{eq:natexten1}\\
\noalign{\vskip6pt}
|w|^{n-1} \int_{\real^n\times\real^n\times\real^n} 
\delta \left[ \tau + |z|^2 + |x|^2 - |y|^2\right] 
\delta (w-x-y-z) 
\big[ |z|\, |x|\, |y|\big]^{-(n-1)} dx\,dy\,dz \label{eq:natexten2}
\end{gather}
are uniformly bounded in terms of the variable $\tau>0$ and $w\in \real^n$ with $n>1$.
\bigskip

These results form the basis for stating the principal multilinear embedding theorem. 
But it seems better to formulate this result in less than full generality since roughly 
the average value of $\alpha_k$ should be on the order of $n$.
Hence the embedding result will be stated for the case of uniform potentials 
$\alpha_k = n-1$ which satisfies all the required conditions for 
Theorems~\ref{thm:alpha} and \ref{thm:delta-function}.

\begin{thm}\label{thm:satisfies}
Let $\alpha_k = n-1$  for $k=1$ to $m$, $3\le m\le n+1$, $n\ge 2$, 
$\sigma = 2+n-m$, $r\ge 1$, $1<p<\infty$, $1/p + 1/q =1$, $rq \ge 2$ and 
$1/p_* + 1/(rq) =1$; then 
\begin{gather} 
\bigg[ \int_{\real^n\times\real_+} \bigg[ \int |x_1 + \cdots  + x_m|^{\sigma/p} 
|\hat f|^r \,d\nu\bigg]^q \,dw\,d\tau\bigg]^{p_*/(rq) } \notag\\
\noalign{\vskip6pt}
\le C\, \Lambda_{p_*} \Big( f;\, \{\alpha\}/(pr)\Big) 
\label{eq:satisfies}
\end{gather} 
\end{thm}

\begin{proof}
Apply Theorems~\ref{thm:alpha} and \ref{thm:delta-function} in the initial ``Outline of 
Argument.'' 
\renewcommand{\qed}{}
\end{proof}
\medskip

Letting indices and parameters be given as above in Theorem~\ref{thm:satisfies}, 
then duality gives for $g\in L^p (\real^n\times\real_+)$, $1<p<\infty$:

\begin{cor}
For $D = \{ x_k \in \real^n,\, k=1\text{ to } m : |x_m|^2 \ge \sumprime |x_k|^2\}$
\begin{gather}
\bigg[ \int_D \Big| g\Big( \sum x_k , |x_m|^2 - \sumprime |x_k|^2\Big) \Big|^p 
\Big| \sum x_k\Big|^\sigma \Prod |x_k|^{-(n-1)} \,dx_1 \ldots dx_m\bigg]^{1/p}
\notag\\
\le C\, \bigg[ \int_{\real^n\times\real_+} |g(w,\tau)|^p \,dw\,d\tau\bigg]^{1/p}
\label{eq:cor}
\end{gather}
\end{cor}

Following on the more limited choices of values for the $\{\alpha_k$'s$\}$ in 
Theorem~\ref{thm:satisfies}, it's possible to give an alternative proof with 
these values, e.g. $\alpha_k = n-1$, for Theorem~\ref{thm:delta-function}  using 
dimension reduction to show 
$$\sup \Delta_n (w) \le C \sup \Delta_2 (w)\ .$$
The proof that $\Delta_2 (w)$ is bounded comes easily from the latter part of the 
argument for Theorem~\ref{thm:delta-function}. 

By recasting the potential calculation given in the Appendix for \cite{KM} from 
a spherical surface to a hyperbolic surface, there is an implicit suggestion that the 
natural object of study for convolution forms given over surfaces may correspond 
to the type:
\begin{equation}\label{eq:recast}
\sup_{w,\tau} |w|^p \int \delta \Big( \tau + \mathop{\sumprime}_{k_j}  |x_k|^2 - |x_m|^2\Big)
\delta \Big( w-\sum_{k'_j} x_k\Big) \Prod |x_k|^{-\alpha_k} 
dx_1 \ldots dx_m
\end{equation}
where $\{k_j\}$ and $\{k'_j\}$ may be different sequences. 
Relevance of specific forms will likely reflect application.

Observe from equation~\ref{eq:outline} that for $p_* =2$, then 
$$\Lambda_{p_*} \left( f; \{\alpha\}/pr\right) 
= \int_{\real^n\times \cdots \times\real^n} 
\Prod |\xi_k|^{2\alpha_k/(pr)} |\hat f(\xi)|^2 \, d\xi_1 \ldots d\xi_m$$
and since $p_* = 2= qr$, then $q\le 2$ and $p\ge 2$ with $pr >2$ for $r>1$.
In this case and using the hypothesis of Theorem~\ref{thm:satisfies}, one has 
$$\frac2{pr} \sum \alpha_k 
= \frac2{pr} m (n-1) < (m-1) n \ \text{ for }\ m=r\ ,\ n+1\ .$$
Hence Theorem~\ref{thm:satisfies} gives a multilinear embedding estimate where the 
embedding indices can fall below the critical lower bound expressed in the restriction 
results obtained in \cite{Beckner-MRL}. 
But still the standard Sobolev embedding estimates hold for iterated potentials so in 
that context it is also possible to move below the critical index for restriction. 
For $\alpha_k = \alpha_0$ for all $k$, $0<\alpha_0 <n$, and $p= 2n/(n-\alpha_0)$
$$\Big[ \|f\|_{L^p (\real^{mn})}\Big]^2 
\le c\, \Lambda_2 \{\alpha\}  < c \int_{\real^{mn}} \Big|(-\Delta/4\pi^2)^{m\alpha_0/4} f\Big|^2
\, dx$$

Implicit in the formulation of the problems treated here is the continuing development 
of new forms that characterize control by smoothness for size.
As an example, and a consequence of the principal estimate obtained here, bounds for 
new Stein-Weiss integrals with a kernel determined by restriction to a smooth submanifold 
can be shown.

Starting from the convolution form used in Theorem~\ref{thm:satisfies} 
$$\int \delta \Big( 1+\sumprime |x_k|^2 - |x_m|^2\Big) 
\delta\Big( w-\sum x_k\Big) 
\Prod |x_k|^{-(n-1)} \,dx_1 \ldots dx_m 
\le \frac{C}{|x|^\sigma}$$
for $\sigma = 2+n-m$, $3\le m\le n+1$, and $n\ge 2$, one obtains forms of Stein-Weiss type.

\begin{thm}\label{thm:tau} 
Define for $\tau >0$, and $w,v\in \real^n$
\begin{gather*} 
K_\tau (w,v) = \int_{\real^n\times\cdots\times\real^n} 
\Prod |x_k|^{-(n-1)} \Big[ \big| w-\sum x_k\big|\, \big| v-\sum x_k\big|\Big]^{-(n+m)/2\ +\ 1}
\times\\
\delta \Big[ \tau +\sumprime |x_k|^2 - |x_m|^2\Big] \, dx_1 \ldots d_m\ ,
\end{gather*} 
then for non-negative $f\in L^2 (\real^n)$
\begin{equation}\label{eq:tau} 
\int_{\real^n\times \real^n} f(w) K_\tau (w,v) f(v)\,dw\,dv 
\le c\int_{\real^n} |f|^2\,dx 
\end{equation}
\end{thm}

\begin{proof}
Apply Pitt's inequality and the uniform bounds obtained from 
Theorems~\ref{thm:alpha} and \ref{thm:delta-function}. 
Here $c$ is a generic constant.
\end{proof}

In contrast to regular Stein-Weiss integrals, this kernel $K_\tau$ is not homogeneous 
with respect to dilation. 
By taking $\tau=0$, the resulting kernel is homogeneous and satisfies the conditions 
given for the Stein-Weiss lemma from the Appendix in \cite{Beckner-PAMS08}. 
This allows in principle a formula for the optimal constant. 

\begin{Cor}\label{cor:1}
For the integral operator 
$$(Tf) (w) = \int_{\real^n} K_0 (w,v) f(v)\,dv$$
that maps $L^2 (\real^n)$ to $L^2(\real^n)$
$$\|Tf\|_{L^2 (\real^n)} \le A\|f\|_{L^2 (\real^n)}$$
with the optimal constant given by 
\begin{equation}\label{eq:cor1}
A = \int_{\real^n} K_0 (w,\hat e_1) |w|^{-n/2} \,dw 
\end{equation} 
where $\hat e_1$ is a unit vector in the first coordinate direction.
\end{Cor}

Further extensions come from taking a non-negative integrable function on $\real_+$ 
and integrating out the surface delta function.

\begin{Cor}\label{cor:2} 
For $\varphi \ge0$ with 
\begin{equation*}
K_\varphi (w,v) 
= \!\! \int_{\real^n\times\cdots \real^n} \mkern-18mu
\Prod |x_k|^{-(n-1)}\Big[ \big| w-\sum x_k\big|\, \big| v-\sum x_k\big|\Big]^{-(n+1)/2\ +\ 1} 
\varphi \Big[ |x_m|^2 - \sumprime |x_k|^2\Big]\, dx_1 \ldots dx_m
\end{equation*} 
and $\int_0^\infty \varphi (t)\,dt =1$, then 
\begin{equation}\label{eq:cor2} 
\int_{\real^n\times\real^n} 
f(w) K_\varphi (w,v) f(v)\,dw\,dv \le c\int_{\real^n} |f|^2\,dx\ .
\end{equation} 
\end{Cor}

Overall these results are examples of the following general embedding estimate which 
gives ``size control'' at the $L^2$ level:
\begin{equation}\label{eq:size-ctrl} 
\int_{S\subset \real^{mn}} \Big| f\Big(\sum x_k\Big)\Big|^2 \Prod |x_k|^{-\alpha_k} \,d\nu
\le c \int_{\real^n} \big| (-\Delta/4\pi^2)^{\rho/2} f\big|^2\,dx 
\end{equation}
where $d\nu$ is surface measure (possibly weighted) on the surface $S$ contained 
in $\real^{mn}$.

Practical application for the Klainerman-Machedon method and such convolution-type 
estimates has proved to be efficient by replacing the Riesz potentials with Bessel 
potentials on the Fourier transform side (\cite{CP}, \cite{Kirk}); 
advantage is achieved by removing local singularities while gaining integrability 
on the potential side and improving the range of application as ``smoothing operators''; 
still the lack of homogeneity limits determination of precise dependence on parameters in 
computing best size estimates. 
But as with exact model calculations, the role of Riesz potentials can result in ``very 
elegant and useful formulae'' that underline intrinsic geometric structure, capture 
essential features of symmetry and uncertainty, and provide insight to precise 
lower-order effects.


\end{document}